\documentclass{gtmon_a}
\pdfoutput=1

\usepackage{xypic}

\proceedingstitle{Proceedings of the Nishida Fest (Kinosaki 2003)}
\conferencestart{28 July 2003}
\conferenceend{8 August 2003}
\conferencename{International Conference in Homotopy Theory}
\conferencelocation{Kinosaki, Japan}

\editor{Matthew Ando}
\givenname{Matthew}
\surname{Ando}

\editor{Norihiko Minami}
\givenname{Norihiko}
\surname{Minami}

\editor{Jack Morava}
\givenname{Jack}
\surname{Morava}

\editor{W Stephen Wilson}
\givenname{W Stephen}
\surname{Wilson}

\title{Divisibility of characteristic numbers}

\author{Simone Borghesi}
\givenname{Simone}
\surname{Borghesi}
\address{Universit\'a degli Studi di Milano-Bicocca\\
Dipartimento di Matematica e Applicazioni\\\newline
via Cozzi 53\\
20125 Milano\\
Italy}
\email{mandu2@libero.it}
\urladdr{}

\volumenumber{10}
\issuenumber{}
\publicationyear{2007}
\papernumber{3}
\startpage{63}
\endpage{74}

\doi{}
\MR{}
\Zbl{}

\arxivreference{}


\received{9 September 2004}
\revised{14 May 2005}
\accepted{}
\published{29 January 2007}
\publishedonline{29 January 2007}
\proposed{}
\seconded{}
\corresponding{}
\version{}

\makeatletter
\def\cnewtheorem#1[#2]#3{\newtheorem{#1}{#3}[section]
\expandafter\let\csname c@#1\endcsname\c@theorem}

\let\xysavmatrix\xymatrix
\def\xymatrix{\disablesubscriptcorrection\xysavmatrix}
\AtBeginDocument{\let\bar\wbar\let\tilde\wtilde\let\hat\what}

\makeop{Spec}
\makeop{top}
\makeop{decomposables}
\makeop{pt}

\newtheorem{theorem}{Theorem}[section]
\cnewtheorem{proposition}[theorem]{Proposition}
\cnewtheorem{lemma}[theorem]{Lemma}
\cnewtheorem{corollary}[theorem]{Corollary}
\theoremstyle{remark}
\cnewtheorem{definition}[theorem]{Definition}
\cnewtheorem{example}[theorem]{Example}
\cnewtheorem{remark}[theorem]{Remark}

\makeatother  

\newcommand{\q}{{\mathbb Z}/q}
\newcommand{\p}{{\mathbb Z}/p}

\newcommand{\am}{{\mathcal{A}}^{**}}

\newcommand{\eilmac}{{\bf H}_{\q}}
\newcommand{\eilmacz}{{\bf H}_{\mathbb{Z}}}

\newcommand{\nin}{\not\kern-4pt\in}

\newcommand{\MU}{{\bf MU}}

\newcommand{\mgl}{{\bf MGl}}
\newcommand{\ph}{\Phi}

\begin{document}

\begin{abstract}
We use homotopy theory to define certain rational coefficients
 characteristic numbers with integral values, depending on a given
prime number $q$ and positive integer $t$. We prove the first
nontrivial {\it degree formula} and use it to show that existence
of morphisms between
algebraic varieties for which these numbers are not divisible by $q$
give information on the degree of such morphisms or on zero cycles of
the target variety.
\end{abstract}

\maketitle


\section{Introduction}

Given a complex vector bundle $V$ on a smooth, compact complex
manifold $M$ of dimension $m$, we can associate a collection of
integers $\{n_I\}_I$ which are bundle isomorphisms invariant.
These numbers are obtained as follows: let $I=(i_1,i_2,\ldots,
i_r)$ and $f_I(c_j)=c_1^{i_1}(V)c_2^{i_2}(V)\ldots c_r^{i_r}(V)$
be a monomial of degree $m$ in the Chern classes of $V$  (we set
$c_j$ to have degree $2j$). Then $n_I$ is defined to be $\langle
f_I(c_j), [M]_{\eilmacz}\rangle$, that is the Kronecker pairing of
$f_I(c_j)$
 with the integral coefficients fundamental homology class of
 $M$. We will deal with numbers associated with two special bundles: one is
 the tangent bundle $T_M$, and the other is the normal bundle $\nu_i$ to a
 closed embedding $i\co M\hookrightarrow\mathbb{R}^N$ for $N$ large
 enough.  By the Thom--Pontryagin construction, if $[M]$ denotes the
complex cobordism
 class of $M$, $\MU$  the complex cobordism spectrum and $h$ is
its Hurewicz homomorphism $\pi_*(\MU)\to H_*(\MU,\mathbb{Z})$,
then the coefficients of the polynomial generators of
$H_*(\MU,\mathbb{Z})$ in the expression of $h[M]$ are certain
numbers $\langle
s_{(\alpha_1,\alpha_2,\ldots,\alpha_r)}(\nu_i),[M]_{\eilmacz}\rangle$
associated to $\nu_i$. These numbers are sometimes referred to as
{\it the characteristic numbers of $M$}. Notice that, by
definition, $T_M+\nu_i=0$ in $K_0(M)$, the Grothendieck group of
vector bundles on $M$. Thus such numbers depend just on $M$ (in
fact on the complex cobordism class of $M$). Since the image of
$h$ is completely known (see Stong \cite{stong}), we get information on such
numbers, for instance, if
$s_{(\alpha_1,\alpha_2,\ldots,\alpha_r)}(M)$ is the coefficient of
$b_1^{\alpha_1}b_2^{\alpha_2}\ldots b_r^{\alpha_r}\in
H_*(\MU,\mathbb{Z})$ in $h[M]$, then  we know that $p$ always
divides $s_{(0,0,\ldots,1)}(M)$, where the $1$ entry is in the
$(p^t{-}1)$th place, for any prime
number $p$ and positive integer $t$. 
We will denote such an integer as $s_{p^t-1}(M)$.
More in general, for each
choice of polynomial generators of $\pi_*(\MU)$, the coefficients
of $[M]$ are integers and can be expressed in function of
characteristic numbers of $M$ itself, which can be written as a
rational linear combination of Chern numbers (eg expression
\eqref{divis} for algebraic surfaces). This way one gets
divisibility properties of Chern numbers, much like,
by the Grothendieck--Riemann--Roch formula, we deduce divisibility
properties of the Todd numbers. In a sense, such divisibility is
maximal and the varieties which make it maximal enjoy surprising
properties. In the case of
$s_{p^t-1}(M)$, Voevodsky remarked that
the algebraic varieties for which that number is {\it not}
divisible by $p^2$ are very special. Given a smooth algebraic
variety $X$ over a field $k$ we can talk about {\it characteristic
numbers} by considering the polynomials in the Chern classes of
the tangent bundle of $X$,
$s_{(\alpha_1,\alpha_2,\ldots,\alpha_r)}(\nu_i)=s'_{(\alpha_1,\alpha_2,
\ldots,\alpha_r)}(T_X)$. They are represented by zero dimensional
cycles, that is integral linear combinations $z=\sum_un_uz_u$,
where $z_u\co \Spec k_u\to X$ are (closed) points of $X$ (that is, $k_u$
are finite field extensions of $k$). An integer can be associated
to  $z$: the {\it degree} of $z$ (see \fullref{zerocycle}),
denoted with $\deg(z)$. The integer $\deg s'_{(\alpha_1,\alpha_2,
\ldots,\alpha_r)}(T_X)$ plays the role of $s_{(\alpha_1,\alpha_2,
\ldots,\alpha_r)}(M)$. This statement can be made more precise: if
$X$ admits a closed embedding $X\hookrightarrow 
\mathbb{P}^N$ (such varieties are said {\it projective}) and $k$ admits an
embedding in the complex numbers, then $\deg
s'_{(\alpha_1,\alpha_2,\ldots,\alpha_r)}(T_X)$ equals
$s_{(\alpha_1, \alpha_2,\ldots,\alpha_r)}(X(\mathbb{C}))$, where
$X(\mathbb{C})$ is the topological space defined by the same
equations as $X$ except that we see them having coefficients in
the complex numbers by means of one of the embeddings of $k$ in
$\mathbb{C}$. It can be proved that these integers do not depend
on the field embedding. Under these assumption on the base field
$k$, Voevodsky noted that, if $p^2$ does not divide
$s_{p^t-1}(Y)$ and $f\co Y\to X$ is any
algebraic morphism to a smooth, projective variety $X$ of nonzero
dimension less or equal than $Y$, one of the following must hold:
(i) $f$ is surjective and $\dim(X)=\dim(Y)=p^t-1$ or (ii) there
exists a closed point $\Spec L\to X$ with $p$ not dividing $[L:k]$.
This result follows from a sharper statement, now known as {\it
degree formula}: keeping the same assumptions, we have that
$(1/p)s_{p^t-1}(Y)\equiv\deg(f)(1/p)s_{p^t-1}(X)\mod p$, or else
$X$ has a point not divisible by $p$. This manuscript is based on
a talk I gave at the 2003 International Conference of Algebraic Topology
held in Kinosaki, Japan
in honor of Goro Nishida's 60th birthday. We will go through Voevodsky's
original idea which gave rise to such formula by using homotopy
theory to define characteristic numbers with rational coefficients
and discuss Rost's use of this formula on quadrics proving results
originally due to Hoffmann and Izhboldin (see
Merkurjev~\cite{rost-deg-for}).
Degree formulae involving each a different set of characteristic
numbers, with various {\it obstruction ideals} associated to them,
have also been derived by Rost (top Segre numbers)
\cite{rost-deg-for} and Merkurjev~\cite{mer} over any field, by
the author in \cite{io-grado} over perfect fields and by Levine
and Morel~\cite{lev-mor} over fields of characteristic zero. All
of them can be used to yield proofs of the results of Hoffmann and
Izhboldin and as indicated by Rost.

\section{Characteristic numbers and the spectrum $\ph_1$}

Voevodsky defined a cohomology theory \cite{voe-rosso} on algebraic
varieties $H^{*,*}(-,A)$ which is expected to play the role
of singular cohomology in topology, as prescribed by the {\it
  Beilinson conjectures}. This cohomology theory is called {\it motivic
cohomology} with coefficients in an abelian group $A$, and, when
evaluated on an algebraic variety, it is equipped
of a tautological structure of left module over the (bistable) motivic
cohomology operations $\am_m$. In the case $A=\p$ and perfect base
field, Voevodsky proved
that $\am_m$ contains the bigraded  Hopf $\p$ algebra $H^{*,*}(\Spec  k,
\p)\otimes_{\p}\smash{\mathcal{A}^{*,*}_{\top}}:=\am$ where
$\smash{\mathcal{A}^{*,*}_{\top}}$ is the topological Steenrod algebra with an
appropriate bigrading \cite{voe-reduced}. In particular, for a fixed prime
$p$, we have the left multiplication
$$Q_t\cdot\co H^{*,*}(-,\p)\to H^{*+2p^t-1,*+p^t-1}(-,\p)$$
by the Milnor operation $Q_t\in\mathcal{A}^{*,*}_{\top}$. In our approach
we will use the triangulated category
$\mathcal{SH}(k)$ that shares some good properties with the ordinary
stable homotopy category. Adopting Voevodsky's definition of
motivic cohomology, if $k$ is a perfect field, $H^{*,*}(-,A)$ is a representable functor in $\mathcal{SH}(k)$. The
representing object is denoted by $\mathbf{H}_A$ and is called the
{\it motivic Eilenberg--MacLane spectrum with coefficients in $A$}.
A fundamental fact of this cohomology theory is that, if $k$ is a
perfect field and $X$ is a smooth scheme, $H^{2*,*}(X,A)$ is
isomorphic to the Chow group $CH^*(X)\otimes_{\mathbb{Z}} A$.
 To any algebraic
variety $X$ we can associate a graded ring $CH^*(X)$. Each homogeneous
component $CH^n(X)$ is defined as the quotient of the free group over
the codimension $n$ closed subvarieties in $X$ modulo {\it rational
equivalence}. The product
of two classes $Z_1$ and $Z_2$ is represented by the closed
subscheme obtained as the intersection of two representatives of $Z_1$
and $Z_2$ intersecting {\it properly}. If $\mathbb{P}V$ is the
projectivization of a
rank $n+1$ vector bundle $V\to X$, then
\begin{equation}
CH^*(\mathbb{P}V)\cong\dfrac{CH^*(X)[t]}{(p(t))}
\end{equation}
as left $CH^*(X)$ algebra, where $p(t)$ is a monic degree $n$
polynomial. If $p(t)=t^n+c_1t^{n-1}+\cdots +c_{n-1}t+c_n$, we can let
$c_i$ to be the $i$th Chern class of $V$, which will be therefore
represented by some integral coefficient linear combination of
codimension $i$ closed subvarieties of $X$.

In algebraic geometry there is an analogue concept of what it is known
as {\it characteristic numbers} in homotopy theory.
To a $n$--tuple of nonnegative integers $\alpha=(\alpha_1,\ldots,
\alpha_n)$ with
$0\leq\alpha_1\leq\alpha_2\leq\cdots\leq\alpha_n$
we associate the symmetric polynomial $g_{\alpha}(\mathbf{t})$
in the variables $t_1,\ldots, t_z$ for $z\geq\sum\alpha_i$.
This is the unique symmetric polynomial having monomials
with $\alpha_1$ variables raised to the power $1$,
$\alpha_2$ variables to the power $2$ and so on. Let
$f_\alpha(\sigma_1,\sigma_2,
\ldots,\sigma_n)$ be the polynomial in the elementary
symmetric functions $\sigma_i(\mathbf{t})$ such that
$f_\alpha(\mathbf{t})=g_\alpha(\mathbf{t})$.
\begin{definition}\label{s-alpha}
Let $M$ be an algebraic variety of pure dimension $m$ and
$V\to M$ be a vector bundle. For a $n$--tuple of nonnegative integers
$\alpha=(\alpha_1,\ldots,\alpha_n)$ such that $\sum_{j=1}^nj\alpha_j=m$,
\begin{enumerate}
\item  set $s_\alpha(c_1,\ldots,c_n)=
f_\alpha(\sigma_1,\ldots,\sigma_n)$
by formally replacing the variables $\sigma_i$ with $c_i(V)$, the $i$th
Chern class of $V$.
\item The zero cycle
$s_\alpha(V)$ is defined as $s_\alpha(c_1(V),
c_2(V),\ldots c_n(V))$.
\end{enumerate}
\end{definition}

\begin{definition}\label{zerocycle}
Given a zero dimensional cycle $z$ in an algebraic variety $X$ over a
field $k$, that is, an
integral coefficients linear combination
$$\sum_{i=1}^n\lambda_i\Spec L_i$$ of closed points $\Spec L_i\to X$, we
define the {\it degree} of $z$ to be the integer
$\sum_{i=1}^n\lambda_i[L_i:k]$.
\end{definition}

We recall that to each prime number $q$ and positive integer $t$,
there is a canonical motivic cohomological operation $Q_t$ of
degree $(2q^t-1,q^t-1)$.
In the early version of Voevodsky's proof of the Milnor Conjecture
\cite{voe-mil}, the operation $Q_t$ appeared implicitly in an argument
employing the homology theory $(\ph_1)_{*,*}(-)$ represented
by the object defined by the exact triangle $$\ph_1\to\eilmac
\stackrel{Q_t}{\to}\Sigma^{2q^t-1,q^t-1}\eilmac$$
The spectrum $\ph_1$
is related to the number
$(1/q)s_{q^t-1}$.  To show this, we are going to consider the
spectrum $\mgl\in\mathcal{SH}(k)$, formally defined in the same
way as in homotopy theory: $\mgl_n$ is the Thom space of the
universal $n$--plane bundle over the infinite grassmanian Gr$_n$.
The definition of the structure morphisms is slightly more subtle
than in homotopy theory. The spectrum $\mgl$ shares the same nice
ring object properties as $\MU$ does in classical homotopy theory. In
particular, for any smooth, projective variety $X$ of pure
dimension $n$, there exists a canonical {\it fundamental homology
class} $[X]_\mgl\in \mgl_{2n,n}(X)$ (see \cite{io-tesi} or
\cite{io-grado}). The graded group
$H^{*,*}(\mgl,\q)$ has been entirely computed in \cite{io} on
characteristic zero fields and in \cite{io-perf}, more generally, on
perfect fields. Since $H^{i,j}(\mgl,\q)$ is zero for $i>2j$, for each $r$,
there exists a map $e_1\co \mgl\to \ph_1$, lifting the {\it
Thom class}: $\mgl\to \eilmac$.

\begin{theorem}\label{bir}
Let $i\co k\hookrightarrow\mathbb{C}$ be  field and
$p\co X\to \Spec k$ be the structure morphism of a projective, smooth
variety $X$ of pure dimension. For each
prime number $q$ and positive integer $t$, let $\ph_1$ be the
corresponding spectrum. Then there exists a choice of $e_1$ such that
$$(e_1)_*p_*[X]_\mgl=
\begin{cases}
\deg(X) & \mathrm{if}\; \dim(X)=0\\
-\tfrac{1}{q}\deg s_{q^t-1}(T_X) & \mathrm{if}\; \dim(X)=q^t-1\\
0 & \mathrm{otherwise}
\end{cases}
$$
\end{theorem}
\begin{proof}
Instead of considering $X$, we can argue on the algebraic
variety $X_{\mathbb{C}}^i$ given by the limit of the diagram
$$
\xymatrix{ & X\ar[d]^p\\\Spec \mathbb{C}\ar[r]^{i^*} & \Spec k}
$$
In fact, for any weighted degree $n$ homogeneous polynomial $a(x_u)$
(that is, we set the degree of $x_u$ to be $2u$) and any field embedding
$e\co k\hookrightarrow L$, we have
$\deg(a(c_u(T_X)))=\deg(a(c_u(T_{X_L})))$, and, in particular, this
integer does not depend on the field embedding $e$. The advantage of
considering $X_{\mathbb{C}}^i$ is that this algebraic variety has a
canonical complex manifold associated: $X^{\top}:=X(\mathbb{C},i)$, that has
the points $\Spec \mathbb{C}\to X_{\mathbb{C}}^i$ as underlying set. We
recall that under these assumptions there is a {\it topological
realization functor} (see Morel and Voevodsky \cite{morvoe}) which
induces morphisms
$t_{\mathbb{C}}\co \mgl_{2*,*}(X)\to\MU_{2*}(X^{\top})$ and $t_{\mathbb{C}}
\co CH^*(X_{\mathbb{C}}^i)\otimes A=H^{2*,*}(X,A)\to H^{2*}(X^{\top},A)$.
This morphisms are compatible with the Kronecker product in the sense
that $\langle z, [X_{\mathbb{C}}^i]_{CH}\rangle=\langle t_{\mathbb{C}}z,
[X^{\top}]_H\rangle$ for $z\in CH_0(X_{\mathbb{C}}^i)$. Since
$\langle \smash{s_\alpha(T_{X_{\mathbb{C}}^i})}, [X_{\mathbb{C}}^i]_{CH}\rangle
=\deg(\smash{s_\alpha(T_{X_{\mathbb{C}}^i})})$, this
shows that the usual topological characteristic numbers $\langle
s_\alpha, [X^{\top}]_H\rangle$ are independent on the field embedding
$i\co k\hookrightarrow \mathbb{C}$ and equal to the {\it algebraic}
characteristic numbers in \fullref{s-alpha}. Let now fix a
choice of $e_1\co \mgl\to\ph_1$ lifting the Thom class which induces
$e_1^{\top}\co \MU\to\ph_1^{\top}$. The morphism
$t_{\mathbb{C}}\co (\ph_1)_{2*,*}(\Spec \mathbb{C})\to
(\ph_1^{\top})_{2*}(\pt)$ is an isomorphism for any $*$, thus we will
consider $p_*[X^{\top}]_{\MU}\in\MU_{2*}(\pt)$. This ring is isomorphic
to the Lazard ring $\mathbb{Z}[1,a_1,a_2,\ldots]$ with generators that
can be chosen to have the following properties: let $h$ be the
Hurewicz homomorphism, if $i\neq p^u-1$ for any prime $p$ and positive
integer $u$, then $h(a_i)=b_i+\cdots$ and if $i=p^u-1$,
$h(a_i)=p(b_i+\cdots)$, where $b_i$ are the indecomposable elements of
$H_{2*}(\MU,\mathbb{Z})$. Using the classical Thom
Pontryagin construction and the above discussion on characteristic
numbers, we see that the coefficients of
$b_\alpha=b_1^{\alpha}b_2^{\alpha_2}\ldots b_m^{\alpha_m}$ in $h(p_*
[X^{\top}]_{\MU})$ are precisely $\deg s_\alpha(-T_{X_{\mathbb{C}}})$, where
$-T_{X_{\mathbb{C}}}$ is the opposite virtual bundle of $T_{X_{\mathbb{C}}}$ in
$K_0(X_{\mathbb{C}})$. Since $(\ph_1^{\top})_{2*}(\pt)=0$ unless $*=0$
and $*=q^t-1$, we may assume $\dim(X)=n=q^t-1$, because the statement
of \fullref{bir} in the case of zero dimensional $X$ is
essentially tautological. If $\dim(X)=q^t-1$, then $p_*[X^{\top}]_{\MU}=m
a_{q^t-1}+\decomposables$, for an integer $m$. $h(p_*[X^{\top}]_{\MU})=
mqb_{q^t-1}+\cdots$ and $(e_1^{\top})_*a_{q^t-1}=1\in
(\ph_1^{\top})_{2(q^t-1)}(\pt)\cong\q$ so
$(e_1^{\top})_*p_*[X^{\top}]_{\MU}=(1/q)s_{q^t-1}(-T_X)+
 (e_1^{\top})_*(\decomposables)\mod q$. The decomposable elements of
$\MU_{2(q^t-1)}(\pt)$ are divided in two classes: those whose
Hurewicz image is not divisible by $q$ and the others. If $f(a_i)$
is a monomial of the first kind, then $(e_1^{\top})_*f(a_i)$ can be
any value of $\q$, if the lifting $e_1$ is chosen appropriately.
In fact, by adding a suitable map
$\MU\smash{\stackrel{c}{\to}}\Sigma^{2(q^t-1)}\smash{\eilmac^{\top}}\to
\smash{\ph_1^{\top}}$ to our choice of $\smash{e_1^{\top}}$, we can make
$\smash{(e_1^{\top})_*}f(a_i)$ to be an arbitrary value. Just set
$c\co \MU\to\Sigma^{2(q^t-1)}\smash{\eilmac^{\top}}$ to be a multiple of the
dual of any $b_\alpha$, whose coefficient in the expression of
$h(f(a_i))$ is not divisible by $q$. In the case $f(a_i)$ is
decomposable and $h(f(a_i))$ is divisible by $q$, then
$(e_1)_*f(a_i)=0$ for any choice of $e_1$. Thus, we have a choice
of $e_1^{\top}$ with the property that
$(e_1^{\top})_*p_*[X^{\top}]_{\MU}
=(1/q)s_{q^t-1}(-T_X)\mod q$. Since
$t_{\mathbb{C}}\co (\ph_1)_{2*,*}(\Spec \mathbb{C})\to
(\ph_1^{\top})_{2*}(\pt)$ is an isomorphism, to get the same result
in the algebraic setting it suffices to consider the lifting
$e_1+c$. Finally, the equality
$s_{q^t-1}(-T_X)=-s_{q^t-1}(T_X)$ follows from the fact that
$s_{i}$ are dual to indecomposable classes. This
finishes the proof of the \fullref{bir}.
\end{proof}

The existence of a topological realization functor in the case of the
base field $k$ embedding in $\mathbb{C}$ enables us to use
constructions we know to exist in topology, such as the
Thom--Pontryagin. If we wish to prove \fullref{bir} over a finite
or, more in general, perfect field $k$, then the scheme of the proof
is similar, but we cannot shift the argument to topology and we are
forced to exclusively use algebraic homotopy categories. This involves
the developement of such constructions of topological origin
in this algebraic setting. The details of them are
written in \cite{io-grado}.

In general, there is no canonical choice of the polynomial
generators $a_i$ of $\MU_{2*}(\pt)$, but $(e_1)_*p_*[X]_\mgl$ does
not depend on such choices, thus these numbers should be
expressible only in function of $e_1$. This dependence appears just
in presence of $a_i$ with the property that $h(a_i)$ is not
divisible by the prime number $q$ we are considering. The classes
$a_i$ can be all chosen to be $p_*[M]_\MU$ for certain smooth
projective complete intersections $M$. It follows that the
coefficients of $h(a_i)$ are explicitely computable and so do the
numbers $(e_1)_*p_*[X]_\mgl$ for each choice of $e_1$. For
instance, working at the prime $q=3$ and $t=1$, we are looking at
$\MU_4(\pt)$ which is generated as free abelian group by the
classes $a_2$ and $a_1^2$ where $a_2$ can be canonically chosen as
$p_*[M^2]_\MU$ for any smooth cubic surface $M^2$ in
$\mathbb{P}^3$ and $a_1$ as $p_*[M^1]_\MU$ for any smooth quadric
curve in $\mathbb{P}^2$. Carring out the computations we get, if
$X$ is a smooth, projective algebraic variety over a perfect
field,
\begin{equation}
(e_1)_*p_*[X]_\mgl=
\begin{cases}
\deg(X)\mod 3 & \text{if} \dim(X)=0\\
-\bigl(\tfrac{1}{3}-\tfrac{\lambda}{2}\bigr)\deg s_{(0,1)}(T_X)\\
\qua+\tfrac{\lambda}{4}\deg\bigl(s_{(1)}^2(T_X)-s_{(2)}(T_X)\bigr)\mod 3 &
\text{if X is a surface}\\ 0  & \text{otherwise}
\end{cases}
\end{equation}
where $\lambda$ can be any integer. This parameter reflects the
dependence of these numbers by the choice of the lifting $e_1$. In
particular, setting $\lambda=0$, we recover the formula of \fullref{bir}. This unusual way of getting characteristic numbers of
smooth and projective varieties, implicitly provides information on
divisibility of characteristic numbers. In fact, any number
obtained with this construction is necessarily a reduction modulo
$q$ of a linear combination with rational coefficients of
characteristic numbers with integral values. In the above example,
we conclude that
\begin{equation}\label{divis}
-(1/3)+\lambda/2)\deg s_{(0,1)}(T_X)+(\lambda/4)
(\deg(s_{(1)}^2(T_X)-s_{(2)}(T_X))
\end{equation}
 is always an
integer for any integer $\lambda$. For $\lambda=0$ we get that
$\deg s_{(0,1)}(T_X)$ is always divisible by $3$ for any smooth
projective surface over a characteristic zero field.

\section{Divisibility of characteristic numbers and degree formulae}

In a sense, the $\mathbb{Q}$ linear relations between such characteristic
numbers are {\it maximal} among those with integral value: for
instance, let fix a characteristic zero field $k$; if we consider the
rational numbers $(1/qi)\deg
s_{q^t-1}(T_X)$ for each integer $i\neq 1$
or $-1$ there exists a smooth, projective algebraic variety over $k$
such that $(1/qi)\deg s_{q^t-1}(T_X)$ is
a nonintegral rational number. The smooth, projective algebraic
varieties $Y$, with the property that
$(e_1)_*p_*[Y]_\mgl$ is not further divisible by the prime number $q$
in question, are of very special kind. The existence of any (rational)
morphism $f$ from such $Y$ to any smooth projective variety $X$ gives
information on algebraic cycles of $X$ and the degree of $f$. Indeed
one of the following two statements must hold: (i)
the number $(e_1)_*p_*[X]_\mgl$ is nonzero and the degree of $f$ is
not divisible by $q$, or (ii) $X$ has
a closed point $x$ with $[k(x):k]$ not divisible by $q$. A {\it
  rational morphism} $Y\to X$ is a map $U\to X$ where $U\subset X$ is
a dense open set. This is a
consequence of the so called {\it degree formula}: let $t_n(W)$ denote
$(e_1)_*p_*[W]_\mgl$ for a smooth, projective algebraic variety $W$ of
pure dimension $n$,
\begin{theorem}\label{ciao}
Let $f$ be a rational morphism $Y\longrightarrow X$
between smooth, projective algebraic varieties of pure dimension
$n$  over a characteristic zero field $k$. Then the relation
\begin{equation}
t_n(Y)=(\deg f) t_n(X)
\end{equation}
holds in $(\q)/I(X)$, where $I(X,q)=\langle [k(x):k]\mod
q\backslash x: \text{ closed point of } X\rangle$ if
$\dim(Y)=\dim(X)=q^t-1$, and $I(X,q)=0$ otherwise.
\end{theorem}

\begin{proof}
We first assume that the morphism $f$ is regular.
Applying the homology long exact sequences originating
from the exact triangle
\begin{equation}
\xymatrix{\Sigma^{2(q^t-1),q^t-1}\eilmac\ar[r] & \ph_1\ar[r] &\eilmac}
\end{equation}
to the commutative diagram
\begin{equation}
\xymatrix{Y\ar[r]^f\ar[dr]^{p_Y}& X\ar[d]^{p_X}\\
& \Spec k}
\end{equation}
we obtain the commutative diagram of groups:
{\small
\begin{equation}
\renewcommand{\theequation}{\normalsize\arabic{equation}}
\begin{split}
\xymatrix{\cdots \to CH_{*-q^t+1}(Y)\otimes\q\ar[r]\ar[d]^{f_*} &
  (\ph_1)_{2*,*}(Y)\ar[r]\ar[d]^{f_*} &
  CH_*(Y)\otimes\q\to\cdots\ar[d]^{f_*}\\
\cdots \to CH_{*-q^t+1}(X)\otimes\q\ar[r]\ar[d]^{{p_X}_*} &
  (\ph_1)_{2*,*}(X)\ar[r]\ar[d]^{{p_X}_*} &
  CH_*(X)\otimes\q\to\cdots\ar[d]^{{p_X}_*}\\
  CH_{*-q^t+1}(\Spec k)\otimes\q\ar[r] &
  (\ph_1)_{2*,*}(\Spec k)\ar[r] & CH_*(\Spec k)\otimes\q}
\end{split}
\end{equation}
}
The class $f_*(e_1)_*[Y]_\mgl-(\deg f) (e_1)_*[X]_\mgl\in
(\ph_1)_{2n,n}(X)$ lies in the image of some class in
$CH_{n-q^t+1}(X)\otimes\q$, because $f_*[Y]_{CH}=(\deg f)[X]_{CH}$ and
$(e_1)_*[W]_\mgl=[W]_{CH}$ for any choice of $e_1$. On the other hand,
$${p_X}_*(f_*(e_1)_*[Y]_\mgl-(\deg f) (e_1)_*[X]_\mgl)=t_n(Y)-(\deg
f)t_n(X)$$ and ${p_X}_*\co CH_0(X)\otimes\q\to CH_0(\Spec k)\otimes\q$ has
$I(X,q)$ as image. Since $CH_{*-q^t+1}(\Spec k)\otimes\q\to
  (\ph_1)_{2*,*}(\Spec k)$ is always injective, the statement of
\fullref{ciao} for regular morphisms follows. The case of rational
morphism is treated by decomposing the rational morphism $f$ as the
diagram
$$\xymatrix{Y'\ar[d]^g\ar[dr]^{f'} & \\
Y\ar@{.>}[r]^f & X}
$$
where both $g$ and $f'$ are regular morphisms and $g$ is a birational
equivalence. Then we apply the regular version of the theorem to $g$
and $f'$ and use that $I(X)$ is birational invariant for smooth
projective varieties.
\end{proof}

\fullref{bir} along with the description of the numbers given in
\fullref{ciao}, yields a formula
refined enough to prove some results about quadrics derived by Rost
(see Merkurjev \cite{rost-deg-for}).

\begin{theorem}
Let $Q$ be a smooth projective quadric of dimension $d\geq 2^m-1$ and
let $X$ be a smooth variety over $k$ such that $I(X,2)=0$ and
admitting a rational map $f\co \xymatrix{Q\ar@{.>}[r] & X}$. Then,
\begin{enumerate}
\item $\dim(X)\geq 2^m-1$;
\item if $\dim(X)=2^m-1$, there is a rational map
$\xymatrix{X\ar@{.>}[r] & Q}$.
\end{enumerate}
\end{theorem}

\begin{proof}
Choose a subquadric $Q'\subset Q$ of dimension $2^m-1$ so
that the restriction of $f$ to $Q'$ is a rational morphism. We will
start with proving the theorem by assuming that
$t_{2^m-1}(Q')\neq 0\mod 2$ and then we will show this
assertion. For the first part we can use Rost's proof, which we
rewrite here. If
$\dim(X)<2^m-1$, then we consider the composition $f'$
\begin{equation}
\xymatrix{Q'\ar@{.>}[r]^f & X\ar[r] & X\times
  \mathbb{P}_k^{2^m-1-\dim(X)}}
\end{equation}
to which we can apply the degree formula of \fullref{ciao}. Notice
that $I(X)=I(X\times\mathbb{P}^i)$ for any $i$. Since $\deg(f')=0$,
to show the first statement, it suffices to prove that
$t_{2^m-1}(Q')\neq 0$. The same formula implies that
$\deg(f)$ is odd. In the case the dimension of $X$ equals $2^m-1$,
this can be rephrased by saying that the quadric
$Q''=Q'\times_k \Spec k(X)$ has a point over $k(X)$ of odd
degree (namely its generic point as variety over $k(X)$). In turn,
this is equivalent for $Q''$ to become isotropic over an odd degree
field extension of $k(X)$. By Springer's Theorem then $Q''$ is
isotropic over $k(X)$, that is, there exists a rational morphism
$g\co \xymatrix{X\ar@{.>}[r] & Q''}$. The rational morphism of the second
statement of the theorem is the composition
$$\xymatrix{X\ar@{.>}[r]^g & Q''\ar[r]^{\mathrm{can}} &
Q'\ar@{^(->}[r]& Q}$$
To prove that $t_{2^m-1}(Q')\neq 0$ for
a smooth, projective quadric $Q'$, we use \fullref{bir} and the
properties of the classes $s_{i}$: (a)
$s_{i}(A+ B)=s_{i}(A)+
s_{i}(B)$ for two virtual bundles $A$ and $B$
and (b) $s_{i}(L)=c_1(L)^i$ for a
line bundle $L$. By definition, the quadric $Q'$ is determined by a
degree two homogeneous equation in $2^m$ variables, thus its tangent
bundle fits in the short exact sequence
\begin{equation}
0\to T_{Q'}\to j^*T_{\mathbb{P}^{2^m}}\to j^*\mathcal{O}(2)\to 0
\end{equation}
where $j\co Q'\hookrightarrow\mathbb{P}^{2^m}$ is the closed embedding.
Property (a) implies that
$$s_{i}(T_{Q'})=s_{i}(j^*T_{\mathbb{P}^{2^m}})-
s_{i}(j^*\mathcal{O}(2)).$$
Property (b) yields
$$s_{i}(j^*\mathcal{O}(2))=j^*
s_{i}(\mathcal{O}(2))=j^*c_1^i(\mathcal{O}(2))=j^*2^ic_1^i
\mathcal{O}(1)=2^iQ'\cap c_1^i\mathcal{O}(1).$$
To compute $s_{i}(j^*T_{\mathbb{P}^{2^m}})$ we use the
short exact sequence of vector bundles over $\mathbb{P}^{2^m}$
$$0\to T_{\mathbb{P}^{2^m}}\to\oplus_{2^m+1}\mathcal{O}(1)\to
\mathcal{O}\to 0$$
We get
\begin{multline*}
s_{i}(j^*T_{\mathbb{P}^{2^m}})=s_{i}
(j^*\oplus_{2^m+1}\mathcal{O}(1))-s_{i}(j^*\mathcal{O}) \\
= (2^m+1)j^*c_1^i\mathcal{O}(1)=(2^m+1)Q'\cap c_1^i\mathcal{O}(1).
\end{multline*}
Let now $i=2^m-1$ and compute the degree of the zero cycle that we
get
$$\deg(s_{2^m-1}(T_{Q'}))=(2^m+1)\deg(Q'\cap
c_1^i\mathcal{O}(1))-2^{2^m-1}\deg(Q'\cap c_1^i\mathcal{O}(1))$$
Since $c_1\mathcal{O}(1)$ is the class of an hyperplane in
$\mathbb{P}^{2^m}$, we have that
$\deg(Q'\cap c_1^i\mathcal{O}(1))=2$. Hence we conclude that
\begin{align*}
t_{2^m-1}(Q') &=-(1/2)\deg(s_{2^m-1}(T_{Q'}))\mod 2 \\
&=-(2^m+1-2^{2^m-1})\mod 2 \\
&=1\mod 2
\end{align*}
which completes the proof.
\end{proof}

\begin{corollary}[Hoffmann]
Let $Q_1,Q_2$ be two anisotropic quadrics.
If $\dim(Q_1)\geq
2^m-1$ and $Q_2$ is isotropic over $k(Q_2)$, then $\dim(Q_2)\geq
2^m-1$.
\end{corollary}

\begin{corollary}[Izhboldin] Let $Q_1,Q_2$ be two anisotropic
  quadrics. If $\dim(Q_1)=\dim(Q_2)=2^m-1$ and $Q_2$ is isotropic over
  $k(Q_1)$, then $Q_1$ is isotropic over $k(Q_2)$.
\end{corollary}

We may wonder now, if this process of obtaining degree formulae
may be generalized to other spectra.
The answer to this question is positive: the
spectra $\ph_r$ for $r>1$ constructed on characteristic zero fields
\cite{io} and on perfect fields \cite{io-perf}, can be used to
define new rational characteristic numbers for each $r$ and choice of lifting
$e_r\co \mgl\to \ph_r$, each of them appearing in certain degree formulae
\cite{io-grado}.

Summing up, we first used homotopy theory to define certain numbers
$(e_1)_*p_*[X]_\mgl$, expressed them as rational characteristic
numbers with integral values (\fullref{bir}) and showed that they
satisfy a degree formula (\fullref{ciao}). Lastly, we computed
them for smooth quadrics and used this degree formula to prove a
result of Rost, at least under the assumption on the base field $k$ of
being of characteristic zero.

\bibliographystyle{gtart}
\bibliography{link}

\end{document}